\documentclass[11pt,a4paper]{amsart}
\usepackage[T1]{fontenc}
\usepackage[utf8]{inputenc}
\usepackage[english]{babel}
\usepackage{amsmath,amssymb,amsthm,mathrsfs}
\usepackage{hyperref}

\newtheorem{theorem}{Theorem}
\newtheorem{lemma}{Lemma}

\newtheorem{cor}{Corollary}
\let\originalforall=\forall
\renewcommand{\forall}{\mathop{\vcenter{\hbox{\Large$\originalforall$}}}}

\let\originalexists=\exists

\renewcommand{\exists}{\mathop{\vcenter{\hbox{\Large$\originalexists$}}}}
\title[Inversion in algebras of integrable functions]{Inversion problem in algebras of integrable functions with summable Fourier transforms}
\author{Przemysław Ohrysko}
\address{University of Warsaw, Institute of Mathematics, Banacha 2A, 02-097 Warsaw, Poland}
\email{p.ohrysko@gmail.com}
\begin{document}
\maketitle
\begin{abstract}
    In this paper, we study the norm-controlled inversion problem in two classes of algebras of integrable functions. In contrast of the classical case of $L^{1}(G)$, we prove that this problem has a positive solution in our setting without any additional restrictions.
\end{abstract}
\section{Introduction}

\subsection{Preliminaries in Banach algebra theory}

Let $A$ be a commutative complex Banach algebra with identity $\mathbf{1}$.
The \emph{maximal ideal space} (or \emph{Gelfand space}) of $A$ is denoted by $\Delta(A)$ and consists of all
nonzero continuous homomorphisms $\varphi:A\to\mathbb C$.
For $a\in A$ the \emph{Gelfand transform} is the function
\[
\widehat a:\Delta(A)\to\mathbb C,\qquad \widehat a(\varphi)=\varphi(a).
\]
The \emph{spectrum} of $a$ is
\[
\sigma_{A}(a)=\{\lambda\in\mathbb C:\ a-\lambda\mathbf{1}\ \text{is not invertible in }A\}.
\]
In the commutative setting one has $\sigma_{A}(a)=\widehat a(\Delta(A))$, and the \emph{spectral radius}
is given by $r(a)=\lim_{n\to\infty}\|a^{n}\|^{1/n}$.

If $A$ is a (commutative) Banach algebra without identity, its standard unitization is
$A^{1}=A\oplus\mathbb C\mathbf{1}$ with multiplication
\[
(a,\lambda)\cdot(b,\mu)=(ab+\lambda b+\mu a,\lambda\mu),
\]
and a norm equivalent to $\|(a,\lambda)\|=\|a\|+|\lambda|$.
It is worth mentioning that if $A$ is a commutative Banach algebra, then the Gelfand space of $A^{1}$ is the one-point compactification of $\Delta(A)$ and the multiplicative-linear functional $\varphi_{\infty}$ corresponding to the point at infinity is defined by the formula $\varphi_{\infty}(a+\lambda\mathbf{1})=\lambda$.

\subsection{Controlled inversion problem}

A recurring theme in harmonic analysis and the theory of Banach algebras is the
\emph{invertibility} of elements and, more quantitatively, \emph{norm-controlled inversion}:
given a Banach algebra $A$ (typically commutative) and a distinguished “visible” part $X$
of the maximal ideal space, one seeks bounds of the form
\[
\|f^{-1}\|_{A}\ \le\ \Phi(\delta)
\quad\text{whenever}\quad
\|f\|_{A}\le 1,\ \ \inf_{\varphi\in X}|\widehat f(\varphi)|\ge \delta,
\]
where $\widehat f$ is the Gelfand transform.
In his paper \cite{Nikolski1999}, Nikolski developed a systematic approach to such questions,
with an emphasis on the interplay between the \emph{visible spectrum} (detected by $\widehat f|_{X}$)
and the possible presence of an \emph{invisible part} of the spectrum that may still influence
quantitative resolvent bounds and the solvability of B\'ezout equations.

The present project aims at developing an analogous quantitative theory for the Fourier-type algebras
$A_{p}(G)$ and $L_{p}(G)$ (for compact abelian group $G$)
and for their \emph{unitizations} $A_{p}(G)^{1}=A_{p}(G)\oplus \mathbb C\mathbf{1}$, $L_{p}(G)^{1}=L_{p}(G)\oplus\mathbb C$.
Here $G$ is a locally compact abelian (LCA) group; the Banach algebras $A_{p}(G)$ will be defined later in Section~\ref{sec:ApG}.
Following \cite{Nikolski1999}, for a chosen $X\subset \Delta(A)$ we define the \emph{norm-controlled inversion constant}
\[
c_{1}^{A,X}(\delta)=
\sup\Bigl\{\|f^{-1}\|_{A}:\ f\in A\ \text{invertible},\ \|f\|_{A}\le 1,\
\inf_{\varphi\in X}|\widehat f(\varphi)|\ge \delta\Bigr\}.
\]
The main problem is to determine for which $\delta>0$ the above quantity is finite.
\subsection{Banach algebras $A_{p}(G)$}\label{sec:ApG}

Let $G$ be an LCA group and let $\widehat G$ denote its Pontryagin dual.
For $1\le p<\infty$ Larsen--Liu--Wang \cite{LarsenLiuWang1964} consider the space
\[
A_{p}(G)=\bigl\{f\in L^{1}(G):\ \widehat f\in L^{p}(\widehat G)\bigr\},
\qquad
\|f\|_{A_{p}}:=\|f\|_{1}+\|\widehat f\|_{p}.
\]
Here $\widehat f$ is the Fourier transform on $G$ (in the standard Pontryagin duality normalization).
This definition makes $A_{p}(G)$ a natural “Fourier-regular” ideal inside $L^{1}(G)$.

The basic result is that $A_{p}(G)$ is stable under convolution and complete with respect to the above norm.

\begin{theorem}[Theorem~3 in \cite{LarsenLiuWang1964}]
Let $1\le p<\infty$. Then $A_{p}(G)$ is a Banach algebra under convolution.
\end{theorem}

In particular, for $f,g\in A_{p}(G)$ one has $f*g\in A_{p}(G)$ and the norm is submultiplicative.

A second fact is the identification of the maximal ideal space of $A_{p}(G)$ with the dual group $\widehat G$.
Consequently, the Gelfand transform is (canonically) the Fourier transform.

\begin{theorem}[Theorem~4 in \cite{LarsenLiuWang1964}]
Let $1\le p<\infty$. The maximal ideal space of $A_{p}(G)$ can be canonically identified with $\widehat G$.
Under this identification, for $f\in A_{p}(G)$ the Gelfand transform coincides with $\widehat f$ on $\widehat G$.
\end{theorem}
Since in $A^{1}$ we have $\sigma(x)=\widehat{x}(\Delta(A^{1}))$ the following Corollary is an obvious consequence of the last theorem and the Riemann-Lebesgue lemma.
\begin{cor}
   Let $G$ be a non-discrete LCA group and let $f=\alpha\mathbf{1}+f$, where $\alpha\in\mathbb{C}$ and $f\in A_{p}(G)$ satisfy:
    \begin{equation*}
        |\widehat{x}(\gamma)|=|\alpha+\widehat{f}(\gamma)|\geq\delta>0\text{ for every }\gamma\in\widehat{G},
    \end{equation*}
    then $f$ is invertible in the unitization of $A_{p}(G)$.
\end{cor}
\section{Main results for $A_{p}(G)$}

Let $A=A_{p}(G)$ and let $A^{1}=A_{p}(G)^{1}$ be its unitization. Note that $A_{p}(G)$ is non-unital unless $G$ is discrete and therefore the norm-controlled inversion problem is naturally set in $A^{1}$. The case of discrete $G$ is of independent interest, as for $G=\mathbb{Z}$ we obtain the famous Wiener algebra $A(\mathbb{T})$ with a different (but equivalent) norm. This was one of the main problems studied in \cite{OW2020} extending the results from \cite{Nikolski1999}. Throughout the paper, we assume that $G$ is non-discrete.
\subsection{Initial information}

Since $G$ is non-discrete, $\widehat G$ is non-compact, so the
Riemann--Lebesgue lemma applies: for every $h\in L^{1}(G)$ one has
$\widehat h\in C_{0}(\widehat G)$.

Let $x=\lambda\mathbf 1+a\in A_{p}(G)^{1}$ with $a\in A_{p}(G)\subset L^{1}(G)$.
Then $\widehat x(\gamma)=\lambda+\widehat a(\gamma)$ for $\gamma\in\widehat G$, and $\widehat a\in C_{0}(\widehat G)$.
Hence for every $\varepsilon>0$ there exists $\gamma\in\widehat G$ with $|\widehat a(\gamma)|<\varepsilon$, and therefore
\[
|\widehat x(\gamma)|\ge |\lambda|-|\widehat a(\gamma)|>|\lambda|-\varepsilon.
\]
Consequently, if $\inf_{\gamma\in\widehat G}|\widehat x(\gamma)|\ge \delta$, letting $\varepsilon\downarrow 0$ yields
\[
|\lambda|\ge \delta.
\]
Once $|\lambda|\ge\delta>1/2$ and $\|x\|_{A^{1}}\le 1$ are known, Nikolski's splitting lemma yields the norm control (see \cite[Lemma~1.4.3]{Nikolski1999}):

\begin{theorem}
If $x=\lambda\mathbf 1+f\in A^{1}$ satisfies
$1/2<\delta\leq |\lambda|\leq \|x\|\leq 1$, then $x$ is invertible and $\|x^{-1}\|_{A^{1}}\le (2\delta-1)^{-1}$.
\end{theorem}

In particular, for non-discrete $G$, if $x\in A_{p}(G)^{1}$ satisfies $\|x\|_{A^{1}}\le 1$ and
$\inf_{\gamma\in\widehat G}|\widehat x(\gamma)|\ge \delta>1/2$, then $\|x^{-1}\|_{A^{1}}\le (2\delta-1)^{-1}$.
\\
Quite surprisingly, considering for a moment the unitization of $L^{1}(G)$ this result is the best possible - the norm-controlled inversion does not hold for $\delta\leq\frac{1}{2}$ (check Corollary 3.2.5 in \cite{Nikolski1999}). In the next few subsections we will prove that the norm of the inverse in the unitization of $A_{p}(G)$ can be bounded by a function depending only on $\delta$ for every $\delta>0$ which illustrates a difference between these two algebras in spite of their superficial similarity. Our main tool from Fourier analysis is the theorem of Hausdorff and Young which holds for every locally compact abelian group, but we state it for compact abelian groups as it will be our setting.

\begin{theorem}[Hausdorff--Young inequalities]\label{thm:HY}
Assume that $G$ is a compact abelian group and that Haar measure is normalized so that $m(G)=1$.
Let $\widehat G$ be the discrete dual group equipped with counting measure.
\begin{enumerate}
\item[(i)]  If $1\le p\le 2$ and $f\in L^{p}(G)$, then its Fourier transform
$\widehat f=(\widehat f(\gamma))_{\gamma\in\widehat G}$ belongs to $\ell^{p'}(\widehat G)$ and
\[
\|\widehat f\|_{\ell^{p'}(\widehat G)}\le \|f\|_{L^{p}(G)},
\qquad \frac1p+\frac1{p'}=1.
\]
\item[(ii)] (Inverse Hausdorff--Young.) If $1\le r\le 2$ and $b=(b(\gamma))_{\gamma\in\widehat G}\in \ell^{r}(\widehat G)$,
then the Fourier series
\[
(\mathcal F^{-1}b)(t):=\sum_{\gamma\in\widehat G} b(\gamma)\,\gamma(t)
\]
converges in $L^{r'}(G)$ (where $1/r+1/r'=1$) and satisfies
\[
\|\mathcal F^{-1}b\|_{L^{r'}(G)}\le \|b\|_{\ell^{r}(\widehat G)}.
\]
\end{enumerate}
\end{theorem}

\noindent
Both estimates are standard; see, for example, \cite[Theorem~4.28]{Folland2016} for the forward inequality on locally compact abelian groups and \cite{Katznelson2004} for the compact abelian case (Fourier series).

\subsection{Bounds for every $\delta>0$ when $1\le p\le 2$}

Assume now that $G$ is \emph{compact} and normalize Haar measure by $m(G)=1$.
We are going to prove the following theorem.
\begin{theorem}\label{pierw}
    Let $G$ be a compact abelian group and let $x\in A_{p}(G)$ with $p\in [1,2]$ satisfy $\|x\|_{A^{1}}\leq 1$ and $|\widehat{x}(\gamma)|\geq\delta>0$ for every $\gamma\in\widehat{G}$. Then 
    \begin{equation*}
        \|x^{-1}\|_{A^{1}}\leq\frac{1}{\delta}+\frac{2(1-\delta)}{\delta^{2}}.
    \end{equation*}
\end{theorem}
\begin{proof}
Let $1\le p\le 2$ and write $x=\lambda\mathbf 1+f\in A_{p}(G)^{1}$ with $f\in A_{p}(G)$ and $\|x\|_{A^{1}}\leq 1$.
If
\[
\inf_{\gamma\in\widehat G}|\widehat x(\gamma)|\ge \delta>0,
\]
then by the Riemann--Lebesgue lemma $\widehat f\in c_{0}(\widehat G)$ and hence $|\lambda|\ge \delta$ (as above).
Let $y=\frac{1}{\lambda}\mathbf{1}+g$ with $g\in A_{p}(G)$ be the inverse of $x$. Then
\begin{align*}
    \widehat{g}(\gamma)
    &=\frac{1}{\lambda+\widehat{f}(\gamma)}-\frac{1}{\lambda}
     =-\frac{\widehat{f}(\gamma)}{\lambda(\lambda+\widehat{f}(\gamma))}
    \qquad \text{for every }\gamma\in\widehat{G}.
\end{align*}
It follows that $|\widehat{g}(\gamma)|\leq\delta^{2}|\widehat{f}(\gamma)|$ which gives 
\begin{equation}\label{pros}
    \|\widehat{g}\|_{\ell p}\leq \frac{\|\widehat{f}\|_{\ell p}}{\delta^{2}}\leq\frac{1-\delta}{\delta^{2}}
\end{equation}
as $|\lambda|+\|f\|_{L^{1}(G)}+\|\widehat{f}\|_{\ell p}\leq 1$.
\\
By the Hausdorff-Young inequality $g\in L^{q}(G)$ where $\frac{1}{p}+\frac{1}{q}=1$ with $\|g\|_{L^{q}(G)}\leq \|\widehat{g}\|_{\ell p}$. Since $G$ is compact and $m(G)=1$ we get by H\"older inequality and (\ref{pros}):
\begin{equation*}
    \|g\|_{L^{1}(G)}\leq \|g\|_{L^{q}(G)}\leq\|\widehat{g}\|_{\ell p}\leq\frac{1-\delta}{\delta^{2}}.
\end{equation*}
Finally,
\begin{align*}
    \|y\|_{A^{-1}}
    &=\frac{1}{|\lambda|}+\|g\|_{A_{p}(G)}
     =\frac{1}{|\lambda|}+\|g\|_{L^{1}(G)}+\|\widehat g\|_{\ell^{p}(\widehat G)}\\
    &\le \frac{1}{\delta}+\frac{1-\delta}{\delta^{2}}+\frac{1-\delta}{\delta^{2}}
     =\frac{1}{\delta}+\frac{2(1-\delta)}{\delta^{2}}.
\end{align*}
\end{proof}

\subsection{A $p>2$ extension for $\delta>1/3$}

Extending Theorem \ref{pierw} to $p>2$ is non-trivial as there is no variant of Hausdorff-Young inequality for this range. The idea of the next proof is to raise an element $f\in A_{p}(G)$ to sufficiently large power $n$ to get $f^{\ast n}\in A_{q}(G)$ where $q\leq 2$. Next, some algebra is involved since we have to deal with the sums of the form $x=\lambda\mathbf{1}+f$, for which the power trick does not work immediately. In view of Theorem \ref{thm:remove-delta-third} the following result is superfluous, but we keep it as it does not utilize the involution and therefore might be useful in other context.
\begin{theorem}\label{thm:pgt2-delta-third}
Let $G$ be a compact abelian group, let $p>2$, and let $x=\lambda\mathbf 1+f\in A_{p}(G)^{1}$ with
$\|x\|_{A^{1}}\le 1$. Assume that
\[
\inf_{\gamma\in\widehat G}|\widehat x(\gamma)|\ge \delta
\qquad (\gamma\in\widehat G)
\]
for some $\delta>1/3$. Let $n\ge 1$ be an odd integer such that $q:=p/n\le 2$, and set
\[
r:=\frac{1-\delta}{2\delta}\in(0,1),\qquad 
\eta_{n}:=1-r^{n}.
\]
Then $x$ is invertible in $A^{1}$ and one has the explicit estimate
\begin{equation}\label{eq:pgt2-explicit-bound}
\|x^{-1}\|_{A^{1}}
\le
\left(
\delta^{-n}+\frac{2(1-\delta)^{n}}{\delta^{2n}\eta_{n}}
\right),
\end{equation}
\end{theorem}

\begin{proof}
Write $x=\lambda\mathbf 1+f$ with $f\in A_{p}(G)$ and set $u:=\widehat f\in \ell^{p}(\widehat G)$.
Since $f\in L^{1}(G)$ and $G$ is compact, the Riemann--Lebesgue lemma yields $u\in c_{0}(\widehat G)$.

\smallskip
\noindent\emph{Step 1: lower bounds for $|\lambda|$ and a uniform bound for $u$.}
From $\inf_{\gamma}|\lambda+u(\gamma)|\ge \delta$ and $u(\gamma)\to 0$ we get $|\lambda|\ge \delta$.
Moreover, using the triangle inequality and the normalization $\|x\|_{A^{1}}\le 1$ one obtains
\[
\|u\|_{\infty}\le \frac{1-\delta}{2}.
\]
(In particular, for $\delta>1/3$ we have $r=\|u\|_{\infty}/|\lambda|\le (1-\delta)/(2\delta)<1$.)

\smallskip
\noindent\emph{Step 2: odd powers and reduction to an exponent $\le 2$.}
Let $n$ be an odd number such that $q:=\frac{p}{n}\le 2$. Set
\[
y_{n}:=\lambda^{n}\mathbf 1+f^{*n}\in A_{q}(G)^{1},\qquad q:=\frac{p}{n}\le 2.
\]
We also record the corresponding \emph{norm} estimate, which will be needed to apply Theorem \ref{pierw}.
Since $q=p/n<p$ and $\widehat G$ is discrete (because $G$ is compact), we have the continuous embedding
$\ell^{q}(\widehat G)\subset \ell^{p}(\widehat G)$ and the norm inequality $\|a\|_{\ell^{p}(\widehat G)}\le \|a\|_{\ell^{q}(\widehat G)}$.
As $\widehat{f^{*n}}=u^{n}$, we have
\begin{align*}
\|f^{*n}\|_{A_{q}(G)}
&=\|f^{*n}\|_{L^{1}(G)}+\|u^{n}\|_{\ell^{q}(\widehat G)}\\
&\le \|f\|_{L^{1}(G)}^{n}+\|u\|_{\ell^{p}(\widehat G)}^{n}\\
&\le \bigl(\|f\|_{L^{1}(G)}+\|u\|_{\ell^{p}(\widehat G)}\bigr)^{n}
= \|f\|_{A_{p}(G)}^{n}.
\end{align*}
Hence
\begin{align*}
\|y_{n}\|_{A_{q}(G)^{1}}
&=|\lambda|^{n}+\|f^{*n}\|_{A_{q}(G)}\\
&\le |\lambda|^{n}+\|f\|_{A_{p}(G)}^{n}\\
&\le (|\lambda|+\|f\|_{A_{p}(G)})^{n}
= \|x\|_{A_{p}(G)^{1}}^{n}
\le 1.
\end{align*}

Then $\widehat y_{n}(\gamma)=\lambda^{n}+u(\gamma)^{n}$ and, since $|u(\gamma)|\le \|u\|_{\infty}\le r|\lambda|$,
\[
|\widehat y_{n}(\gamma)|\ge |\lambda|^{n}\bigl(1-r^{n}\bigr)=|\lambda|^{n}\eta_{n}.
\]
Hence $y_{n}$ satisfies the hypotheses of the $q\le 2$ case (Hausdorff--Young), which we now implement directly.

\smallskip
\smallskip
\noindent\emph{Step 3: inversion of $y_{n}$ via Theorem \ref{pierw}.}
Since $q\le 2$, we may apply Theorem \ref{pierw} to the element $y_{n}\in A_{q}(G)^{1}$.
The estimates from Step~2 show that $\|y_{n}\|_{A_{q}(G)^{1}}\le 1$ and
\[
\inf_{\gamma\in\widehat G}|\widehat y_{n}(\gamma)|\ge \delta^{n}\eta_{n}.
\]
Hence $y_{n}$ is invertible in $A_{q}(G)^{1}$ and Theorem~4 provides an explicit bound for $\|y_{n}^{-1}\|_{A_{q}(G)^{1}}$
in terms of $\delta^{n}\eta_{n}$.
In addition,
\[
\|y_{n}^{-1}\|_{A_{p}(G)^{1}}\le \|y_{n}^{-1}\|_{A_{q}(G)^{1}},
\]
so any bound obtained for $\|y_{n}^{-1}\|_{A_{q}(G)^{1}}$ also controls $\|y_{n}^{-1}\|_{A_{p}(G)^{1}}$.

\smallskip
\noindent\emph{Step 4: recovering $x^{-1}$.}
For odd $n$ the identity
\[
(\lambda\mathbf 1+f)*Q_{n}(f)=\lambda^{n}\mathbf 1+f^{*n}=y_{n}
\]
holds, where
\[
Q_{n}(f):=\sum_{k=0}^{n-1}(-1)^{k}\lambda^{\,n-1-k}f^{*k}.
\]
Consequently,
\[
x^{-1}=Q_{n}(f)*y_{n}^{-1}
\quad\text{and}\quad
\|x^{-1}\|_{A_{p}(G)^{1}}\le \|Q_{n}(f)\|_{A_{p}(G)^{1}}\,\|y_{n}^{-1}\|_{A_{p}(G)^{1}}.
\]

\smallskip
\noindent\emph{Estimate for $\|Q_{n}(f)\|_{A_{p}(G)^{1}}$.}
By submultiplicativity of the norm in $A_{p}(G)^{1}$ we have
\[
\|Q_{n}(f)\|_{A_{p}(G)^{1}}
\le \sum_{k=0}^{n-1}|\lambda|^{\,n-1-k}\,\|f\|_{A_{p}(G)}^{k}
\le (|\lambda|+\|f\|_{A_{p}(G)})^{n-1}
=\|x\|_{A_{p}(G)^{1}}^{\,n-1}
\le 1,
\]
where in the last step we used the normalization $\|x\|_{A_{p}(G)^{1}}\le 1$.
Combining this with the bound for $\|y_{n}^{-1}\|_{A_{p}(G)^{1}}$ obtained in Step~3 and the identity
$x^{-1}=Q_{n}(f)*y_{n}^{-1}$ gives the announced estimate for $\|x^{-1}\|_{A_{p}(G)^{1}}$.

This yields \eqref{eq:pgt2-explicit-bound}.

\end{proof}

\subsection{The general case}

Dropping the assumption $\delta>\frac{1}{3}$ requires a symmetrization trick. Recall that $L^{1}(G)$ possesses a symmetric involution given by 
\begin{equation*}
    f^{\ast}(x)=\overline{f(-x)}\text{ for }x\in G\text{ and }f\in L^{1}(G).
\end{equation*}
The same formula works in $A_{p}(G)$. The strategy of the proof of the next theorem is to consider an element $a:=x\ast x^{\ast}$ which has non-negative Fourier transform. This solves the problem of the invertibility of the element $y_{n}$ from the proof of Theorem \ref{thm:pgt2-delta-third}.

\begin{theorem}\label{thm:remove-delta-third}
Let $G$ be a compact abelian group, let $p>2$, and let $x=\lambda\mathbf 1+f\in A_{p}(G)^{1}$ satisfy
\[
\|x\|_{A_{p}(G)^{1}}\le 1\qquad\text{and}\qquad\inf_{\gamma\in\widehat G}|\widehat x(\gamma)|\ge \delta>0.
\]
Then $x$ is invertible in $A_{p}(G)^{1}$. More precisely, let $n$ be the smallest odd integer such that $q:=p/n\le 2$, and set
\[
\Delta_{n}:=\delta^{2n}\bigl(1-(1-\delta^{2})^{n}\bigr).
\]
Then
\begin{equation}\label{eq:remove-delta-explicit-bound}
\|x^{-1}\|_{A_{p}(G)^{1}}
\le
\frac{1}{\delta^{2n}}+\frac{2(1-\Delta_{n})}{\Delta_{n}^{2}}.
\end{equation}
\end{theorem}

\begin{proof}
Write $x=\lambda\mathbf 1+f\in A_{p}(G)^{1}$ and set $u:=\widehat f\in \ell^{p}(\widehat G)$, so that $\widehat x(\gamma)=\lambda+u(\gamma)$ for $\gamma\in\widehat G$.
Since $u(\gamma)\to 0$ as $\gamma\to\infty$ and $\inf_{\gamma}|\lambda+u(\gamma)|\ge \delta$, we have $|\lambda|\ge \delta$.
Also $|\lambda|\le \|x\|_{A_{p}(G)^{1}}\le 1$.

Let $x^{\ast}:=\overline{\lambda}\mathbf 1+f^{\ast}$, where $f^{\ast}(t):=\overline{f(t^{-1})}$.
Then $\widehat{x^{\ast}}(\gamma)=\overline{\widehat x(\gamma)}$ for $\gamma\in\widehat G$.
Set
\[
a:=x*x^{\ast}\in A_{p}(G)^{1}.
\]
For every $\gamma\in\widehat G$ we have
\[
\widehat a(\gamma)=\widehat x(\gamma)\,\overline{\widehat x(\gamma)}=|\widehat x(\gamma)|^{2}\ge \delta^{2}.
\]
Expanding the convolution gives the decomposition
\[
a=|\lambda|^{2}\mathbf 1+k,
\qquad
k:=\lambda f^{\ast}+\overline{\lambda}f+f*f^{\ast}\in A_{p}(G),
\]
and for $\gamma\in\widehat G$,
\[
\widehat k(\gamma)=|\lambda+u(\gamma)|^{2}-|\lambda|^{2}=2\Re(\overline{\lambda}u(\gamma))+|u(\gamma)|^{2}\in\mathbb R.
\]
Moreover, by submultiplicativity in $A_{p}(G)^{1}$,
\[
\|a\|_{A_{p}(G)^{1}}\le \|x\|_{A_{p}(G)^{1}}\|x^{\ast}\|_{A_{p}(G)^{1}}=\|x\|_{A_{p}(G)^{1}}^{2}\le 1,
\]
hence
\[
|\lambda|^{2}+\|k\|_{A_{p}(G)}=\|a\|_{A_{p}(G)^{1}}\le 1.
\tag{A}\label{eq:A}
\]

Now choose $n$ to be the smallest odd integer such that $q:=p/n\le 2$, and define
\[
y_{n}:=|\lambda|^{2n}\mathbf 1+k^{*n}\in A_{q}(G)^{1}.
\]
By \eqref{eq:A} we obtain,
\[
\|y_{n}\|_{A_{q}(G)^{1}}\le (|\lambda|^{2}+\|k\|_{A_{p}(G)})^{n}\le 1.
\]
For $\gamma\in\widehat G$ we have
\[
\widehat y_{n}(\gamma)=|\lambda|^{2n}+(\widehat k(\gamma))^{n}.
\]
Since $\widehat k(\gamma)\in\mathbb R$ and $n$ is odd, the only possible cancellation occurs when $\widehat k(\gamma)<0$.
From $\widehat a(\gamma)=|\lambda|^{2}+\widehat k(\gamma)\ge \delta^{2}$ we get $\widehat k(\gamma)\ge \delta^{2}-|\lambda|^{2}$, hence for $\widehat k(\gamma)<0$,
\[
|\widehat k(\gamma)|=-\widehat k(\gamma)\le |\lambda|^{2}-\delta^{2}.
\]
Therefore, for all $\gamma\in\widehat G$,
\[
|\widehat y_{n}(\gamma)|
\ge
|\lambda|^{2n}-(|\lambda|^{2}-\delta^{2})^{n}
=
|\lambda|^{2n}\Bigl(1-\Bigl(1-\frac{\delta^{2}}{|\lambda|^{2}}\Bigr)^{n}\Bigr).
\]
Using $|\lambda|\ge \delta$ and $|\lambda|\le 1$, we obtain
\[
\inf_{\gamma\in\widehat G}|\widehat y_{n}(\gamma)|
\ge
\delta^{2n}\bigl(1-(1-\delta^{2})^{n}\bigr)
=\Delta_{n}.
\]

Since $q\le 2$, we may apply Theorem~\ref{pierw} to $y_{n}\in A_{q}(G)^{1}$ and obtain that $y_{n}$ is invertible in $A_{q}(G)^{1}$ and
\[
\|y_{n}^{-1}\|_{A_{q}(G)^{1}}
\le
\frac{1}{\delta^{2n}}+\frac{2(1-\Delta_{n})}{\Delta_{n}^{2}}.
\]
As $\ell^{q}(\widehat G)\subset \ell^{p}(\widehat G)$ for $q<p$, we also have $\|y_{n}^{-1}\|_{A_{p}(G)^{1}}\le \|y_{n}^{-1}\|_{A_{q}(G)^{1}}$.

Define
\[
Q_{n}(k):=\sum_{j=0}^{n-1}(-1)^{j}\,|\lambda|^{2(n-1-j)}\,k^{*j}.
\]
As before
\[
(|\lambda|^{2}\mathbf 1+k)*Q_{n}(k)=|\lambda|^{2n}\mathbf 1+k^{*n}=y_{n},
\]
Consequently,
\[
a^{-1}=Q_{n}(k)*y_{n}^{-1}.
\]
Moreover, by \eqref{eq:A} we have $\|Q_{n}(k)\|_{A_{p}(G)^{1}}\le (|\lambda|^{2}+\|k\|_{A_{p}(G)})^{n-1}\le 1$, hence
\[
\|a^{-1}\|_{A_{p}(G)^{1}}\le \|y_{n}^{-1}\|_{A_{p}(G)^{1}}.
\]

Finally, $x^{-1}=x^{\ast}*a^{-1}$ in $A_{p}(G)^{1}$. Therefore,
\[
\|x^{-1}\|_{A_{p}(G)^{1}}
\le
\|x^{\ast}\|_{A_{p}(G)^{1}}\|a^{-1}\|_{A_{p}(G)^{1}}
\le
\|a^{-1}\|_{A_{p}(G)^{1}}
\le
\|y_{n}^{-1}\|_{A_{q}(G)^{1}},
\]
which yields \eqref{eq:remove-delta-explicit-bound}.
\end{proof}

\section{Quantitative inversion in the unitized algebras $L^{p}(G)_{1}$}

Throughout this section $G$ is a fixed compact abelian group with normalized Haar
measure $m(G)=1$. For $1\le p<\infty$ we write $L^{p}(G)_{1}:=\mathbb{C}\,1\oplus
L^{p}(G)$ for the unitization of the convolution algebra $(L^{p}(G),*)$\footnote{$L^{p}(G)$ is a Banach algebra as $\|f\ast g\|_{p}\leq \|f\|_{1}\|g\|_{p}\leq \|f\|_{p}\|g\|_{p}$}, equipped
with the norm
\[
\|\lambda 1+f\|_{L^{p}(G)_{1}}:=|\lambda|+\|f\|_{L^{p}(G)}.
\]
For $\gamma\in\widehat G$ we use the Fourier transform
$\widehat f(\gamma)=\int_{G} f(t)\overline{\gamma(t)}\,dt$, and for $x=\lambda 1+f$
we write $\widehat x(\gamma):=\lambda+\widehat f(\gamma)$.  As before, the
involution is $f^{\ast}(t)=\overline{f(t^{-1})}$ and $(\lambda 1+f)^{\ast}=
\overline\lambda\,1+f^{\ast}$.
\\
Also, we have a simple analogue of Theorems $1$ and $2$.
\begin{theorem}
    Let $1\leq p<\infty$ and let $G$ be a compact abelian group. Then $\Delta(L^{p}(G))=\widehat{G}$ and the Gelfand transform is canonically identified with the Fourier transform. Moreover, if $x=\lambda\mathbf{1}+f$ satisfies 
    \begin{equation*}
        \exists_{\delta>0}\forall_{\gamma\in\widehat{G}}|\widehat{x}(\gamma)|=|\lambda+\widehat{f}(\gamma)|\geq\delta,
    \end{equation*}
    then $x$ is invertible.
\end{theorem}

\begin{proof}
For each $\gamma\in\widehat G$ define
\[
\varphi_{\gamma}(\lambda \mathbf{1}+f):=\lambda+\widehat f(\gamma),
\qquad
\varphi_{\infty}(\lambda \mathbf{1}+f):=\lambda .
\]
Since $\widehat{f*g}(\gamma)=\widehat f(\gamma)\widehat g(\gamma)$, each
$\varphi_{\gamma}$ (and $\varphi_{\infty}$) is a continuous multiplicative
functional on $L^{p}(G)_{1}$.

Conversely, let $\varphi\in\Delta(L^{p}(G)_{1})$. For $\gamma\in\widehat G$ set
$u_{\gamma}(t):=\gamma(t)$, so $u_{\gamma}\in L^{p}(G)$ and one has
$u_{\gamma}*u_{\eta}=\delta_{\gamma,\eta}\,u_{\eta}$ and
$f*u_{\gamma}=\widehat f(\gamma)\,u_{\gamma}$ for every $f\in L^{p}(G)$.
Hence $\varphi(u_{\gamma})^{2}=\varphi(u_{\gamma})$, so
$\varphi(u_{\gamma})\in\{0,1\}$ for all $\gamma$.

If $\varphi(u_{\gamma})=0$ for every $\gamma$, then $\varphi$ vanishes on all
trigonometric polynomials (finite linear combinations of $u_{\gamma}$), and by
their density in $L^{p}(G)$ (see e.g.\ \cite[Ch.~I]{Katznelson2004}) we obtain
$\varphi|_{L^{p}(G)}\equiv 0$. Therefore $\varphi=\varphi_{\infty}$.

Otherwise, choose $\gamma_{0}$ with $\varphi(u_{\gamma_{0}})=1$. Then for any
$f\in L^{p}(G)$,
\[
\varphi(f)=\varphi(f)\varphi(u_{\gamma_{0}})
=\varphi(f*u_{\gamma_{0}})
=\varphi(\widehat f(\gamma_{0})u_{\gamma_{0}})
=\widehat f(\gamma_{0}),
\]
and hence $\varphi=\varphi_{\gamma_{0}}$.

Finally, if $x=\lambda\mathbf{1}+f$ satisfies
$|\widehat x(\gamma)|=|\lambda+\widehat f(\gamma)|\ge\delta$ for all
$\gamma\in\widehat G$, then also $|\lambda|\ge\delta$ because
$\widehat f(\gamma)\to 0$ as $\gamma\to\infty$. Thus
$|\varphi(x)|\ge\delta$ for every $\varphi\in\Delta(L^{p}(G)_{1})$, so
$0\notin\sigma(x)$ and $x$ is invertible.
\end{proof}

The two theorems below provide quantitative bounds for $\|x^{-1}\|_{L^{p}(G)_{1}}$
in terms of $p$ and the spectral gap $\delta:=\inf_{\gamma}|\widehat x(\gamma)|$.
The proofs follow the same symmetrization idea used in the proof of
Theorem~\ref{thm:remove-delta-third}: we first pass to $a:=x*x^{\ast}$ so that
$\widehat a(\gamma)=|\widehat x(\gamma)|^{2}$ is real and bounded away from $0$,
then we split $a=|\lambda|^{2}1+k$ and build the auxiliary elements
$y_{n}=|\lambda|^{2n}1+k^{*n}$.

\subsection{The case $1<p\le 2$}
Our strategy for the range $p\in (1,2]$ is the reduction to the case of $p=2$ which by Plancherel theorem is covered in Theorem \ref{pierw}.
\begin{lemma}[Hausdorff--Young reduction to $L^{2}$] \label{lem:HY-L2}
Let $1<p<2$ and $g\in L^{p}(G)$, where $G$ is a compact abelian group equipped with
normalized Haar measure. Set $p'=p/(p-1)$ and choose an integer $m\ge \lceil p'/2\rceil$.
Then $g^{*m}\in L^{2}(G)$ and
\[
\|g^{*m}\|_{2}\le \|g\|_{p}^{m}.
\]
\end{lemma}

\begin{proof}
By the Hausdorff--Young inequality,
$\|\widehat g\|_{\ell^{p'}(\widehat G)}\le \|g\|_{p}$.
Since $2m\ge p'$, we have $\ell^{p'}(\widehat G)\subset \ell^{2m}(\widehat G)$ and hence
$\|\widehat g\|_{\ell^{2m}}\le \|\widehat g\|_{\ell^{p'}}$.
Using Plancherel and the identity $\widehat{g^{*m}}=(\widehat g)^{m}$, we obtain
\[
\|g^{*m}\|_{2}
= \|(\widehat g)^{m}\|_{\ell^{2}}
= \|\widehat g\|_{\ell^{2m}}^{m}
\le \|\widehat g\|_{\ell^{p'}}^{m}
\le \|g\|_{p}^{m}.
\]
\end{proof}

\begin{theorem}[Quantitative inversion in $L^{p}(G)_{1}$ for $1<p\le 2$]\label{lp1}
Fix $1<p\le 2$ and set $p':=\frac{p}{p-1}$. Let $x=\lambda 1+f\in L^{p}(G)_{1}$ satisfy
\[
\|x\|_{L^{p}(G)_{1}}\le 1,
\qquad\text{and}\qquad
\inf_{\gamma\in\widehat G}|\widehat x(\gamma)|\ge \delta>0.
\]
Let $m$ be the smallest odd integer with $m\ge \lceil p'/2\rceil$, and define
\[
\Delta_{m}:=\delta^{2m}\Bigl(1-(1-\delta^{2})^{m}\Bigr).
\]
Then $x$ is invertible in $L^{p}(G)_{1}$ and
\[
\|x^{-1}\|_{L^{p}(G)_{1}}
\le
\frac{1}{\delta^{2m}}+\frac{2(1-\Delta_{m})}{\Delta_{m}^{2}}.
\]
\end{theorem}

\begin{proof}

\noindent\textbf{Comment (the endpoint $p=2$).} If $p=2$, then $p'=2$ and the choice in the statement gives $m=1$. In this endpoint case the steps based on Lemma~\ref{lem:HY-L2} (used below to force $k^{*m}\in L^{2}$) are unnecessary, since $k\in L^{2}(G)$ already and we may work directly with $y_{1}=|\lambda|^{2}1+k$.

Write $x=\lambda 1+f$. Since $\widehat f(\gamma)\to 0$ as $\gamma\to\infty$ and
$|\lambda+\widehat f(\gamma)|\ge\delta$ for all $\gamma$, we have $|\lambda|\ge\delta$.
Also $\|x\|_{L^{p}(G)_{1}}\le 1$ implies $|\lambda|\le 1$.

Set $a:=x*x^{\ast}$. Then $\widehat a(\gamma)=|\widehat x(\gamma)|^{2}\ge\delta^{2}$.
Expanding $a$ we obtain
\[
a=|\lambda|^{2}1+k,
\qquad
k:=\lambda f^{\ast}+\overline{\lambda}f+f*f^{\ast}\in L^{p}(G).
\]
By submultiplicativity in the unitization,
$\|a\|_{L^{p}(G)_{1}}\le \|x\|_{L^{p}(G)_{1}}^{2}\le 1$, hence
\begin{equation}\label{eq:Lp-k-bound}
|\lambda|^{2}+\|k\|_{p}\le 1.
\end{equation}
Moreover $\widehat k(\gamma)=|\lambda+\widehat f(\gamma)|^{2}-|\lambda|^{2}$ is real
for all $\gamma$.

Choose $m$ as in the statement and define
$y_{m}:=|\lambda|^{2m}1+k^{*m}$. By \eqref{eq:Lp-k-bound} and the same geometric-series
estimate as in Step~2 of Theorem~\ref{thm:remove-delta-third},
\[
\|y_{m}\|_{L^{p}(G)_{1}}\le (|\lambda|^{2}+\|k\|_{p})^{m}\le 1.
\]
Next, for every $\gamma\in\widehat G$,
\[
\widehat y_{m}(\gamma)=|\lambda|^{2m}+(\widehat k(\gamma))^{m}.
\]
Since $m$ is odd and $\widehat k(\gamma)\in\mathbb{R}$, the ``negative'' case is controlled
exactly as in the proof of Theorem~\ref{thm:remove-delta-third}: from
$|\lambda|^{2}+\widehat k(\gamma)=\widehat a(\gamma)\ge\delta^{2}$ we obtain
$|\widehat k(\gamma)|\le |\lambda|^{2}-\delta^{2}$ whenever $\widehat k(\gamma)\le 0$,
and hence
\[
|\widehat y_{m}(\gamma)|
\ge
|\lambda|^{2m}-(|\lambda|^{2}-\delta^{2})^{m}
\ge
\Delta_{m}.
\]
Thus $\inf_{\gamma}|\widehat y_{m}(\gamma)|\ge \Delta_{m}$.

We now pass to $L^{2}$. By the Hausdorff--Young reduction lemma (applied to $k$), the choice
$m\ge \lceil p'/2\rceil$ ensures $k^{*m}\in L^{2}(G)$, hence $y_{m}\in L^{2}(G)_{1}$.
Moreover, Lemma~\ref{lem:HY-L2} yields
$\|k^{*m}\|_{2}\le \|k\|_{p}^{m}$. Hence
\begin{align*}
\|y_{m}\|_{L^{2}(G)_{1}}
&=|\lambda|^{2m}+\|k^{*m}\|_{2}\\
&\le |\lambda|^{2m}+\|k\|_{p}^{m}\\
&\le (|\lambda|^{2}+\|k\|_{p})^{m}
= \|a\|_{L^{p}(G)_{1}}^{m}
\le \|x\|_{L^{p}(G)_{1}}^{2m}
\le 1.
\end{align*}
Applying the $p\le 2$ inversion estimate proved in Theorem \ref{pierw} (with $\delta$ replaced
by $\Delta_{m}$) we obtain by Parseval's identity
\[
\|y_{m}^{-1}\|_{L^{2}(G)_{1}}
\le
\frac{1}{\delta^{2m}}+\frac{(1-\Delta_{m})}{\Delta_{m}^{2}}.
\]
Since $m(G)=1$ and $p<2$, we have the continuous embedding $L^{2}(G)\hookrightarrow L^{p}(G)$ and
$\|h\|_{p}\le \|h\|_{2}$. Therefore
$\|y_{m}^{-1}\|_{L^{p}(G)_{1}}\le \|y_{m}^{-1}\|_{L^{2}(G)_{1}}$.

Finally we recover $a^{-1}$ and then $x^{-1}$ using the same polynomial identity as in
Step~4 of Theorem~\ref{thm:remove-delta-third}. Namely, with
\[
Q_{m}(k):=\sum_{j=0}^{m-1}(-1)^{j}|\lambda|^{2(m-1-j)}\,k^{*j}
\]
we have $(|\lambda|^{2}1+k)*Q_{m}(k)=|\lambda|^{2m}1+k^{*m}=y_{m}$ because $m$ is odd.
Hence $a^{-1}=Q_{m}(k)*y_{m}^{-1}$. From \eqref{eq:Lp-k-bound} we get
$\|Q_{m}(k)\|_{L^{p}(G)_{1}}\le (|\lambda|^{2}+\|k\|_{p})^{m-1}\le 1$, so
$\|a^{-1}\|_{L^{p}(G)_{1}}\le \|y_{m}^{-1}\|_{L^{p}(G)_{1}}$.
Since $x^{-1}=x^{\ast}*a^{-1}$ and $\|x^{\ast}\|_{L^{p}(G)_{1}}=\|x\|_{L^{p}(G)_{1}}\le 1$,
the same bound holds for $\|x^{-1}\|_{L^{p}(G)_{1}}$.
\end{proof}

\subsection{The case $p>2$}

For $p>2$ we have to deal with norms larger then the Hilbert space norm. We bound
the $L^{p}$-norm of the inverse by estimating its Fourier coefficients in $\ell^{p'}$
and then applying the inverse Hausdorff--Young inequality (with exponent $p'\in(1,2)$).

\begin{theorem}[Quantitative inversion in $L^{p}(G)_{1}$ for $p>2$]\label{lp2}
Fix $p>2$ and let $p':=\frac{p}{p-1}\in(1,2)$. Let $x=\lambda 1+f\in L^{p}(G)_{1}$ satisfy
\[
\|x\|_{L^{p}(G)_{1}}\le 1,
\qquad\text{and}\qquad
\inf_{\gamma\in\widehat G}|\widehat x(\gamma)|\ge \delta>0.
\]
Let $n$ be the smallest odd integer with $n\ge p-1$, and define
\[
\Delta_{n}:=\delta^{2n}\Bigl(1-(1-\delta^{2})^{n}\Bigr),
\qquad
c_{n}(\delta):=1-(1-\delta^{2})^{n}.
\]
Then $x$ is invertible in $L^{p}(G)_{1}$ and
\[
\|x^{-1}\|_{L^{p}(G)_{1}}
\le
\delta^{-2n}+\frac{\delta^{-4n}}{c_{n}(\delta)}.
\]
\end{theorem}

\begin{proof}
As in the previous proof, write $x=\lambda 1+f$ and set $a:=x*x^{\ast}=|\lambda|^{2}1+k$ with
$k\in L^{p}(G)$, $\widehat a(\gamma)=|\widehat x(\gamma)|^{2}\ge \delta^{2}$, and
$|\lambda|^{2}+\|k\|_{p}\le 1$. In particular $|\lambda|\ge\delta$ and $|\lambda|\le 1$.

Choose $n$ as in the statement and set $y_{n}:=|\lambda|^{2n}1+k^{*n}$. Exactly as before,
\[
\|y_{n}\|_{L^{p}(G)_{1}}\le (|\lambda|^{2}+\|k\|_{p})^{n}\le 1,
\qquad
\inf_{\gamma}|\widehat y_{n}(\gamma)|\ge \Delta_{n}.
\]
Write
\[
\widehat y_{n}(\gamma)=|\lambda|^{2n}(1+t(\gamma)),
\qquad
t(\gamma):=\Bigl(\frac{\widehat k(\gamma)}{|\lambda|^{2}}\Bigr)^{n}.
\]
Using the pointwise estimate
\[
|\widehat y_{n}(\gamma)|\ge |\lambda|^{2n}-(|\lambda|^{2}-\delta^{2})^{n}
=|\lambda|^{2n}\Bigl(1-\bigl(1-\delta^{2}/|\lambda|^{2}\bigr)^{n}\Bigr),
\]
we obtain for every $\gamma\in\widehat G$
\[
|1+t(\gamma)|
=\frac{|\widehat y_{n}(\gamma)|}{|\lambda|^{2n}}
\ge 1-\bigl(1-\delta^{2}/|\lambda|^{2}\bigr)^{n}
\ge 1-(1-\delta^{2})^{n}=c_{n}(\delta),
\]
where in the last step we used $|\lambda|\le 1$.
Consequently
\begin{align*}
(\widehat y_{n}(\gamma))^{-1}
&=|\lambda|^{-2n}\frac{1}{1+t(\gamma)}
=|\lambda|^{-2n}+b(\gamma),\\
b(\gamma)&:=-|\lambda|^{-2n}\frac{t(\gamma)}{1+t(\gamma)}.
\end{align*}
Since $p'\in(1,2)$, the Hausdorff--Young inequality implies
\begin{align*}
\|\widehat k\|_{\ell^{p}(\widehat G)}
&\le \|k\|_{L^{p'}(G)}
\le \|k\|_{L^{p}(G)}
\le 1.
\end{align*}
Therefore
\[
\|(\widehat k)^{n}\|_{\ell^{p/n}(\widehat G)}
=
\|\widehat k\|_{\ell^{p}(\widehat G)}^{n}
\le 1.
\]
Because $n\ge p-1$ we have $p/n\le p'$, hence $\ell^{p/n}\subset \ell^{p'}$ and
$\|(\widehat k)^{n}\|_{\ell^{p'}(\widehat G)}\le 1$. It follows that
\[
\|t\|_{\ell^{p'}(\widehat G)}
=
\Bigl\|\Bigl(\frac{\widehat k}{|\lambda|^{2}}\Bigr)^{n}\Bigr\|_{\ell^{p'}(\widehat G)}
\le
|\lambda|^{-2n}\|(\widehat k)^{n}\|_{\ell^{p'}(\widehat G)}
\le
\delta^{-2n}.
\]
Using the lower bound $|1+t|\ge c_{n}(\delta)$ we get
\[
\|b\|_{\ell^{p'}(\widehat G)}
\le
|\lambda|^{-2n}\,\frac{1}{c_{n}(\delta)}\,\|t\|_{\ell^{p'}(\widehat G)}
\le
\frac{\delta^{-4n}}{c_{n}(\delta)}.
\]
Now apply the inverse Hausdorff--Young inequality (with exponent $p'\in(1,2)$):
the inverse Fourier series map sends $\ell^{p'}(\widehat G)$ to $L^{p}(G)$ boundedly, and
\[
\| \mathcal{F}^{-1}b\|_{L^{p}(G)}\le \|b\|_{\ell^{p'}(\widehat G)}.
\]
Thus $y_{n}^{-1}=|\lambda|^{-2n}1+h$ with $\|h\|_{p}\le \|b\|_{\ell^{p'}}$, and so
\[
\|y_{n}^{-1}\|_{L^{p}(G)_{1}}
\le
|\lambda|^{-2n}+\|h\|_{p}
\le
\delta^{-2n}+\frac{\delta^{-4n}}{c_{n}(\delta)}.
\]

As in the $1<p<2$ case, we recover $a^{-1}$ and $x^{-1}$ via the polynomial identity.
With $Q_{n}(k):=\sum_{j=0}^{n-1}(-1)^{j}|\lambda|^{2(n-1-j)}k^{*j}$ we have
$a*Q_{n}(k)=y_{n}$ since $n$ is odd, hence $a^{-1}=Q_{n}(k)*y_{n}^{-1}$ and
$\|Q_{n}(k)\|_{L^{p}(G)_{1}}\le (|\lambda|^{2}+\|k\|_{p})^{n-1}\le 1$.
Therefore $\|a^{-1}\|_{L^{p}(G)_{1}}\le \|y_{n}^{-1}\|_{L^{p}(G)_{1}}$, and finally
$x^{-1}=x^{\ast}*a^{-1}$ with $\|x^{\ast}\|_{L^{p}(G)_{1}}\le 1$ gives the same bound for
$\|x^{-1}\|_{L^{p}(G)_{1}}$.
\end{proof}
\section{Concluding remarks and open problems}
\begin{enumerate}
    \item The more general problem than controlled inversion is described in \cite{Nikolski1999} in terms of Bezout equations. In brief, we have $n$ elements $(x_{k})_{k=1}^{n}\in A^{1}$ satisfying 
    \begin{equation*}
        \sum_{k=1}^{n}\|x_{k}\|_{A^{1}}^{2}\leq 1\text{ and }\inf_{\gamma\in\widehat{G}}\sum_{k=1}^{n}|\widehat{x_{k}}(\gamma)|^{2}\geq\delta^{2}>0.
    \end{equation*}
    We ask whether there exists a sequence of elements $(y_{k})_{k=1}^{n}$ in $A^{1}$ such that
    \begin{equation*}
        \sum_{k=1}^{n}x_{k}y_{k}=\mathbf{1}.
    \end{equation*}
    If the answer is affirmative, then the second question is the estimate of the form
    \begin{equation*}
        \sum_{k=1}^{n}\|y_{k}\|^{2}\leq F(\delta),
    \end{equation*}
    where $F$ is some function that depends only on $\delta>0$. However, since both considered classes of algebras are symmetric, Lemma 1.4.2 in \cite{Nikolski1999} gives a positive solution to this problem basing on Theorem \ref{thm:remove-delta-third}, Theorem \ref{lp1} and Theorem \ref{lp2}.
    \item The definition of algebras $A_{p}(G)$ is not restricted to compact abelian groups and the norm-controlled inversion problem is meaningful for all locally compact abelian groups $G$. Unfortunately, our methods cannot be extended even to the simplest possible case of $G=\mathbb{R}$. The main issue is very basic: if $G$ is not compact, then the norms $\|\cdot\|_{p}$ and $\|\cdot\|_{q}$ are incomparable for distinct $p$ and $q$ which prevents the use of most arguments in this paper.
\end{enumerate}

\end{document}